\documentclass[10pt]{amsart}
\usepackage{amssymb,amstext,amsmath,amscd,amsthm,amsfonts,enumerate,graphicx,latexsym,stmaryrd,multicol,geometry}
\usepackage[usenames]{color}
\usepackage[all]{xy}
\newtheorem{thm}{Theorem}[section]
\newtheorem{lem}[thm]{Lemma}
\newtheorem{prop}[thm]{Proposition}
\newtheorem{cor}[thm]{Corollary}
\theoremstyle{definition}
\newtheorem{dfn}[thm]{Definition}
\newtheorem{ques}[thm]{Question}

\newtheorem{rem}[thm]{Remark}
\newtheorem{conv}[thm]{Convention}

\newtheorem{ex}[thm]{Example}

\newtheorem*{claim*}{Claim}
\theoremstyle{remark}
\newtheorem*{ac}{Acknowlegments}
\renewcommand{\qedsymbol}{$\blacksquare$}
\numberwithin{equation}{thm}
\def\A{\mathcal{A}}
\def\a{\mathfrak{a}}
\def\b{\mathfrak{b}}
\def\C{\mathcal{C}}
\def\c{\operatorname{\mathsf{C}}}
\def\cof{\operatorname{\mathsf{Cof}}}
\def\cok{\operatorname{Coker}}
\def\D{\mathrm{D}}
\def\depth{\operatorname{depth}}
\def\Ext{\operatorname{Ext}}
\def\ge{\geqslant}
\def\h{\operatorname{H}}
\def\Hom{\operatorname{Hom}}
\def\ipd{\operatorname{IPD}}
\def\ker{\operatorname{Ker}}

\def\m{\mathfrak{m}}
\def\Mod{\operatorname{\mathsf{Mod}}}
\def\mod{\operatorname{\mathsf{mod}}}
\def\nf{\operatorname{NF}}
\def\p{\mathfrak{p}}
\def\pd{\operatorname{pd}}
\def\res{\operatorname{\mathsf{res}}}
\def\rfd{\operatorname{Rfd}}
\def\sing{\operatorname{Sing}}
\def\spec{\operatorname{Spec}}
\def\supp{\operatorname{Supp}}
\def\syz{\mathrm{\Omega}}
\def\V{\mathrm{V}}
\def\X{\mathcal{X}}
\def\Z{\mathbb{Z}}
\begin{document}
\allowdisplaybreaks
\title[Cofiniteness of local cohomology modules and subcategories of modules]{Cofiniteness of local cohomology modules\\
and subcategories of modules}
\author{Ryo Takahashi}
\address{Graduate School of Mathematics, Nagoya University, Furocho, Chikusaku, Nagoya 464-8602, Japan}
\email{takahashi@math.nagoya-u.ac.jp}
\urladdr{https://www.math.nagoya-u.ac.jp/~takahashi/}
\author{Naoki Wakasugi}
\address{Graduate School of Mathematics, Nagoya University, Furocho, Chikusaku, Nagoya 464-8602, Japan}
\email{naoki.wakasugi.e6@math.nagoya-u.ac.jp}
\thanks{RT was partly supported by JSPS Grant-in-Aid for Scientific Research 19K03443 and 23K03070}
%\dedicatory{}
\begin{abstract}
Let $R$ be a commutative noetherian ring and $I$ an ideal of $R$.
Assume that for all integers $i$ the local cohomology module $\h_I^i(R)$ is $I$-cofinite.
Suppose that $R_\p$ is a regular local ring for all prime ideals $\p$ that do not contain $I$.
In this paper, we prove that if the $I$-cofinite modules forms an abelian category, then for all finitely generated $R$-modules $M$ and all integers $i$, the local cohomology module $\h_I^i(M)$ is $I$-cofinite.
\end{abstract}
\maketitle
%\tableofcontents
%%%%%%%%%%%%%%%%%%%%%%%%%%%%%%%%%%%
\section{Introduction}

Let $R$ be a commutative noetherian ring and $I$ an ideal of $R$.
An {\em $I$-cofinite} $R$-module is by definition an $R$-module $X$ that satisfies both of the following two conditions (a) and (b).
\begin{enumerate}[\qquad(a)]
\item
$\supp X$ is contained in $\V(I)$.
\item
$\Ext_R^i(R/I,X)$ is finitely generated for all integers $i$.
\end{enumerate}
Hartshorne \cite{H} introduced the notion of an $I$-cofinite module, and constructed an example (see Example \ref{1} stated below) where a local cohomology module $\h_I^i(M)$ is not $I$-cofinite, which is a counterexample to a conjecture of Grothendieck \cite{G}.
Since then, so many people have worked on the question asking when $\h_I^i(M)$ is $I$-cofinite, and so many results on it have been obtained; see for example \cite{B} and references therein.

Denote by $\Mod R$ the category of $R$-modules and by $\cof_I(R)$ the full subcategory of $\Mod R$ consisting of $I$-cofinite $R$-modules.
After proving results on the relationship between the categorical structure of $\cof_I(R)$ and the cofiniteness of local cohomology modules, Bahmanpour \cite{B} posed the following question.

\begin{ques}[Bahmanpour]\label{17}
Suppose that $\cof_I(R)$ is an abelian subcategory of $\Mod R$.
Is then $\h_I^i(M)$ an $I$-cofinite $R$-module for all finitely generated $R$-modules $M$ and all integers $i$\,?
\end{ques}

The purpose of this paper is to provide a couple of answers to Question \ref{17} mainly by means of techniques of subcategories of modules.
Denote by $\cof_I^0(R)$ the full subcategory of $\Mod R$ consisting of $R$-modules $X$ satisfying the above condition (b) only; such modules are called {\em $I$-ETH-cofinite} and investigated, see \cite{AB} for example. 
Note that for an $R$-module $M$ and an integer $i$ there are equivalences
$$
\text{$\h_I^i(M)$ is $I$-cofinite}
\ \iff\ \h_I^i(M)\in\cof_I(R)
\ \iff\ \h_I^i(M)\in\cof_I^0(R).
$$
The main result of this paper is the following theorem.

\begin{thm}\label{3}
Assume that one of the following three conditions is satisfied.
\begin{enumerate}[\qquad\rm(1)]
\item
$\cof_I^0(R)$ is abelian.
\item
$\cof_I(R)$ is Serre, and $\h_I^i(R)$ is $I$-cofinite for any integer $i$.
\item
$\cof_I(R)$ is abelian, $\h_I^i(R)$ is $I$-cofinite for any integer $i$, and $\sing R$ is contained in $\V(I)$.
\end{enumerate}
Then $\h_I^i(M)$ is $I$-cofinite for any finitely generated $R$-module $M$ and any integer $i$.
\end{thm}

\noindent
In fact, the singular locus condition in Theorem \ref{3}(3) can be removed, if a suitable assumption is imposed on the finitely generated $R$-module $M$, as follows.

\begin{thm}\label{16}
Suppose that $\cof_I(R)$ is abelian and that $\h_I^i(R)$ is $I$-cofinite for all $i\in\Z$.
Let $M$ be a finitely generated $R$-module such that $\pd_{R_\p}M_\p<\infty$ for all $\p\in\D(I)$.
Then $\h_I^i(M)$ is $I$-cofinite for all $i\in\Z$.
\end{thm}

The organization of this paper is as follows.
In Section 2, we recall the precise definition of an $I$-cofinite module, a theorem of Bahmanpour \cite{B} and Hartshorne's celebrated example.
In Section 3, we first recall the definitions of Serre, abelian and thick subcategories, and then investigate the structures of the subcategories $\cof_I(R),\,\cof_I^0(R)$ of $\Mod R$.
Cases (1) and (2) of Theorem \ref{3} are proved in this section; we apply a Mayer--Vietoris sequence and a result shown in \cite{B0}, respectively.
In Section 4, using the results obtained so far, we explore the structure of the full subcategory consisting of finitely generated modules with $I$-cofinite local cohomology modules.
The proofs of case (3) of Theorem \ref{3} and Theorem \ref{16} are given in this section.
The main methods are the notions of resolving subcategories, large restricted flat dimension, nonfree loci and infinite projective dimension loci.

%%%%%%%%%%%%%%%%%%%%%%%%
\section{Preliminaries}

This short section is to give preliminaries for the later sections.
First of all, we state our convention.

\begin{conv}
Throughout this paper, let $R$ be a commutative noetherian ring and $I$ an ideal of $R$.
All subcategories that are considered in this paper are assumed to be strictly full.
We denote by $\Mod R$ the category of $R$-modules and by $\mod R$ the subcategory of $\Mod R$ consisting of finitely generated $R$-modules.
\end{conv}

Let us recall the definition of a cofinite module, whose notion is a main subject of this paper.

\begin{dfn} 
An $R$-module $M$ is called {\em $I$-cofinite} if the support $\supp M$ of $M$ is contained in $\V(I)$ (or in other words, $M_\p=0$ for each prime ideal $\p$ of $R$ that does not contain $I$) and $\Ext_R^i(R/I, M)$ is a finitely generated $R$-module for all integers $i$ (or equivalently, for all non-negative integers $i$).
We denote by $\cof_I(R)$ the subcategory of $\Mod R$ consisting of $I$-cofinite $R$-modules.
\end{dfn}

The $I$-cofiniteness of a local cohomology module $\h_I^i(M)$ has long been a main topic in the studies of local cohomology.
Bahmanpour \cite[Theorem 3]{B} relates it with the structure of the subcategory $\cof_I(R)$ of $\Mod R$.

\begin{thm}[Bahmanpour]
Consider the following two conditions.
\begin{enumerate}[\quad\rm(i)]
\item
$\h_I^i(M)$ is $I$-cofinite for all finitely generated $R$-modules $M$ and all integers $i$.
\item
$\cof_I(R)$ is an abelian subcategory of $\Mod R$.
\end{enumerate}
If the ring $R$ is semi-local, then {\rm(i)} implies {\rm(ii)}.
\end{thm}

This theorem naturally leads us to ask whether or not (ii) implies (i), which is none other than Question \ref{17} mentioned in the previous section.

Finally, we recall a celebrated example constructed by Hartshorne \cite{H}, and several properties it satisfies.
Note that this example supports Question \ref{17}.

\begin{ex}[Hartshorne]\label{1}
Let $k$ be a field, and let $R=k[x,y][\![u,v]\!]$ be a formal power series ring over a polynomial ring.
Consider the ideal $I=(u,v)$ of $R$ and the finitely generated $R$-module $M=R/(xu+yv)$.
It is shown in \cite[\S3]{H} that the following statements hold.
\begin{enumerate}[\quad(a)]
\item
It holds that $\h_I^i(R)=0$ for all $i\ne2$.
\item
One has $\Hom_R(R/I,\h_I^2(R))\cong R/I$ and $\Ext_R^i(R/I,\h_I^2(R))=0$ for all $i>0$.
\item
The $R$-module $\Hom_R(R/I,\h_I^2(M))$ is not finitely generated.
\item
There is an exact sequence $\h_I^2(R)\xrightarrow{xu+yv}\h_I^2(R)\to\h_I^2(M)\to0$.
\end{enumerate}
Since the support of $\h_I^i(R)$ is contained in $\V(I)$, it follows from (a) and (b) that $\h_I^i(R)$ is $I$-cofinite for all $i$.
The module $\h_I^2(M)$ is not $I$-cofinite by (c).
Hence the subcategory $\cof_I(R)$ of $\Mod R$ is not abelian by (d).
\end{ex}

%%%%%%%%%%%%%%%%%%%%%%%%%%%%%%%%%%%%%%%
\section{The subcategories $\cof_I(R)$ and $\cof_I^0(R)$ of $\Mod R$}

In this section, we give the definitions of the subcategories $\cof_I(R)$ and $\cof_I^0(R)$ of $\Mod R$ which are related to cofiniteness, and investigate the structures of those subcategories.
We start by recalling the definitions of Serre, abelian and thick subcategories of an abelian category, together with their fundamental properties.

\begin{dfn}
Let $\A$ be an abelian category.
Let $\X$ be a subcategory of $\A$.
\begin{enumerate}[(1)]
\item
We say that $\X$ is a {\em Serre subcategory} of $\A$ if it is closed under subobjects, quotient objects and extensions.
\item
We say that $\X$ is an {\em abelian subcategory} of $\A$ if $\X$ inherits the abelian structure from $\A$, namely, if $\X$ is an abelian category with respect to the abelian structure of $\A$.
\item
We say that $\X$ is a {\em thick subcategory} of $\A$ if it is closed under direct summands, extensions, kernels of epimorphisms and cokernels of monomorphisms.
\end{enumerate}
\end{dfn}

\begin{rem}\label{9}
Let $\A$ be an abelian category, and let $\X$ be a subcategory of $\A$.
\begin{enumerate}[(1)]
\item
The subcategory $\X$ is Serre if and only if for a short exact sequence $0\to L\to M\to N\to0$ in $\A$, one has $M\in\X$ if and only if $L,N\in\X$.
\item
The subcategory $\X$ is abelian if and only if it is closed under kernels and cokernels.
\item
Assume that $\X$ is closed under direct summands.
Then $\X$ is a thick subcategory if and only if $\X$ satisfies the two-out-of-three property, that is, $\X$ is such that for a short exact sequence $0\to L\to M\to N\to0$ in $\A$, if two of $L,M,N$ are in $\X$, then so is the third.
\item
Suppose that $\X$ is an abelian subcategory closed under extensions (this is equivalent to supposing that $\X$ is a wide subcategory; see Remark \ref{7} stated below).
Then $\X$ is a thick subcategory.
Indeed, let $M,N\in\A$.
Splicing the split short exact sequences $0\to N\to M\oplus N\to M\to0$ and $0\to M\to M\oplus N\to N\to0$, we get an exact sequence $0\to N\to M\oplus N\to M\oplus N\to N\to0$.
It is easy to observe from this exact sequence that if $\X$ is closed under kernels or cokernels, then it is closed under direct summands.
\item
Suppose that $\X$ is a Serre subcategory.
Then it is obvious that $\X$ is an abelian subcategory.
It follows from (4) that $\X$ is a thick subcategory as well.
\end{enumerate}
\end{rem}

We introduce two subcategories of $\Mod R$ and give some basic properties.

\begin{dfn}
\begin{enumerate}[(1)]
\item
We denote by $\cof_I^0(R)$ the subcategory of $\Mod R$ consisting of modules $M$ such that $\Ext_R^i(R/I, M)$ is finite generated for all integers $i$.
\item
For a subset $\Phi$ of $\spec R$ we denote by $\supp^{-1}\Phi$ the subcategory of $\Mod R$ consisting of modules $M$ such that $\supp M$ is contained in $\Phi$.
\end{enumerate}
\end{dfn}

\begin{rem}\label{8}
\begin{enumerate}[(1)]
\item
Let $\Phi$ be a subset of $\spec R$.
Then $\supp^{-1}\Phi$ is a Serre subcategory of $\Mod R$.
This is an immediate consequence of the fact that for a short exact sequence $0\to L\to M\to N\to0$ of $R$-modules one has $\supp M=\supp L\cup\supp N$.
\item
Clearly, $\cof_I^0(R)$ contains $\mod R$.
It is easy to verify that one has $\cof_I(R)=\cof_I^0(R)\cap\supp^{-1}\V(I)$.
\end{enumerate}
\end{rem}

Here we state a property $\cof_I(R)$ and $\cof_I^0(R)$ share, which is used several times later.

\begin{prop}\label{2}
Both $\cof_I(R)$ and $\cof_I^0(R)$ are thick subcategories of $\Mod R$.
\end{prop}

\begin{proof}
Let $X$ be an $R$-module in $\cof_I^0(R)$ and $Y$ is a direct summand of $X$.
Then for each integer $i$ the $R$-module $\Ext_R^i(R/I,X)$ is finitely generated, and hence so is its direct summand $\Ext_R^i(R/I,Y)$.
It follows that $\cof_I^0(R)$ is closed under direct summands.

Let $0\to L\to M\to N\to0$ be a short exact sequence of $R$-modules.
Then there exists an exact sequence
$$
\Ext_R^{i-1}(R/I,N)\to\Ext_R^i(R/I,L)\to\Ext_R^i(R/I,M)\to\Ext_R^i(R/I,N)\to\Ext_R^{i+1}(R/I,L)
$$
for all integers $i$.
Suppose that $L,M$ belong to $\cof_I^0(R)$.
Then $\Ext_R^i(R/I,M)$ and $\Ext_R^{i+1}(R/I,L)$ are finitely generated for all $i$, and the above exact sequence shows that $\Ext_R^i(R/I,N)$ is also finitely generated for all $i$.
Hence $N$ belongs to $\cof_I^0(R)$.
This shows that $\cof_I^0(R)$ is closed under cokernels of monomorphisms.
In a similar way, one can check that $\cof_I^0(R)$ is closed under kernels of epimorphisms and extensions.

We now conclude that $\cof_I^0(R)$ is a thick subcategory of $\Mod R$.
Combining this with Remarks \ref{9}(5) and \ref{8}, we observe that $\cof_I(R)$ is a thick subcategory of $\Mod R$ as well.
\end{proof}

\begin{rem}\label{7}
A {\em wide subcategory} of an abelian category is by definition a subcategory closed under kernels, cokernels and extensions.
So far, a lot of works have been done on wide subcategories; see \cite{BM,E,HJV,Ho,K,MS,ncoh} for instance.
It follows from Proposition \ref{2} that $\cof_I(R)$ is abelian if and only if it is wide.
\end{rem}

The proposition below is shown by using a theorem given in \cite{B0}.
It says that Question \ref{17} has an affirmative answer if the assumption of abelianity in the question is replaced with the stronger assumption of Serreness and the assumption that $\h_I^i(R)$ is $I$-cofinite for all $i$ is added. 

\begin{prop}\label{11}
Suppose that $\cof_I(R)$ is a Serre subcategory of $\Mod R$ and that $\h_I^i(R)$ is $I$-cofinite for all integers $i$.
Then $\h_I^i(M)$ is $I$-cofinite for all integers $i$ and all finitely generated $R$-modules $M$.
\end{prop}

\begin{proof}
Let $M$ be a finitely generated $R$-module.
We have $\supp M\subseteq\spec R=\supp R$, and by assumption, $\h_I^i(R)$ belongs to $\cof_I(R)$ for all integers $i\ge0$.
Applying \cite[Theorem 2.3]{B0} to the Serre subcategory $\cof_I(R)$, we see that $\h_I^i(M)$ belongs to $\cof_I(R)$ for all integers $i\ge0$.
Thus the proposition follows.
\end{proof}

Now we state and prove the following proposition, which says that Question \ref{17} is affirmative if $\cof_I(R)$ is replaced with $\cof_I^0(R)$.
Note that if $\cof_I^0(R)$ is abelian, then so is $\cof_I(R)$ by Remark \ref{8}(2).

\begin{prop}\label{12}
Suppose that $\cof_I^0(R)$ is an abelian subcategory of $\Mod R$.
Then the $R$-module $\h_I^i(M)$ is $I$-cofinite for all finitely generated $R$-modules $M$ and all integers $i$.
\end{prop}

\begin{proof}
We begin with establishing a claim.
\begin{claim*}
Let $n$ be a non-negative integer.
Let $x_1,\dots,x_n$ be elements of $I$.
Then the module $\h_{(x_1,\dots,x_n)}^i(M)$ belongs to $\cof_I^0(R)$ for all integers $i$.
\end{claim*}
\begin{proof}[Proof of Claim]
We use induction on $n$.
Let $n=0$.
Then $(x_1,\dots,x_n)=(0)$.
We have $\h_{(0)}^0(M)=M$ and $\h_{(0)}^i(M)=0$ for all $i\ne0$.
Hence $\h_{(0)}^i(M)\in\mod R\subseteq\cof_I^0(R)$ for all $i\in\Z$ (see Remark \ref{8}(2)).
Let $n=1$.
Then $(x_1,\dots,x_n)=(x)$, where $x:=x_1$.
Then $\h_{(x)}^i(M)=0$ for all $i\ne0,1$, and $\h_{(x)}^0(M)=\Gamma_{(x)}(M)\in\mod R\subseteq\cof_I^0(R)$.
The \v{C}ech complex of $x$ induces a short exact sequence
$$
0\to M/\Gamma_{(x)}(M)\to M_x\to\h_{(x)}^1(M)\to0
$$
of $R$-modules.
We have $M/\Gamma_{(x)}(M)\in\mod R\subseteq\cof_I^0(R)$.
For each $i$ the module $\Ext_R^i(R/I,M_x)$ is isomorphic to $\Ext_{R_x}^i(R_x/IR_x,M_x)$, and the latter module is zero as $IR_x=R_x$.
Hence $M_x$ is in $\cof_I^0(R)$.
Proposition \ref{2} implies that $\cof_I^0(R)$ is closed under cokernels of monomorphisms, so that $\h_{(x)}^1(M)$ belongs to $\cof_I^0(R)$.

Let $n\ge2$ and fix an integer $i$.
Note that $\h_{\a\cap\b}^j=\h_{\a\b}^j$ for all ideals $\a,\b$ of $R$ and all integers $j$; see \cite[Page 47]{BS}.
Using this fact and the Mayer--Vietoris sequence \cite[3.2.3]{BS}, we get an exact sequence
\begin{align*}
&\h_{(x_1,\dots,x_{n-1})}^{i-1}(M)\oplus\h_{(x_n)}^{i-1}(M)\xrightarrow{f_{i-1}}\h_{(x_1,\dots,x_{n-1})x_n}^{i-1}(M)\to\h_{(x_1,\dots,x_n)}^{i}(M)\\
\to&\h_{(x_1,\dots,x_{n-1})}^{i}(M)\oplus\h_{(x_n)}^{i}(M)\xrightarrow{f_i}\h_{(x_1,\dots,x_{n-1})x_n}^{i}(M)
\end{align*}
for every integer $i$.
The induction hypothesis implies that the $R$-modules
$$
\h_{(x_1,\dots,x_{n-1})}^{i-1}(M),\,
\h_{(x_n)}^{i-1}(M),\,
\h_{(x_1,\dots,x_{n-1})x_n}^{i-1}(M),\,
\h_{(x_1,\dots,x_{n-1})}^{i}(M),\,
\h_{(x_n)}^{i}(M),\,
\h_{(x_1,\dots,x_{n-1})x_n}^{i}(M)
$$
belong to $\cof_I^0(R)$.
From the assumption that $\cof_I^0(R)$ is an abelian subcategory of $\Mod R$, it follows that $\cok f_{i-1}$ and $\ker f_i$ are in $\cof_I^0(R)$.
There is a short exact sequence
$$
0\to\cok f_{i-1}\to\h_{(x_1,\dots,x_n)}^{i}(M)\to\ker f_i\to0
$$
of $R$-modules.
By Proposition \ref{2}, the subcategory $\cof_I^0(R)$ of $\Mod R$ is closed under extensions.
It is seen from the above short exact sequence that the module $\h_{(x_1,\dots,x_n)}^{i}(M)$ belongs to $\cof_I^0(R)$.
\renewcommand{\qedsymbol}{$\square$}
\end{proof}
Let $x_1,\dots,x_n$ be a system of generators of the ideal $I$.
We deduce from the above claim that $\h_I^i(M)$ belongs to $\cof_I^0(R)$ for all integers $i$.
Since the support of the $R$-module $\h_I^i(M)$ is contained in $\V(I)$, we obtain $\h_I^i(M)\in\cof_I^0(R)\cap\supp^{-1}\V(I)=\cof_I(R)$ for all $i$ by Remark \ref{8}(2).
\end{proof}

\begin{rem}
Let $R$ and $I$ be as in Example \ref{1}.
\begin{enumerate}[(1)]
\item
As a consequence of Proposition \ref{11}, one sees that the subcategory $\cof_I(R)$ of $\Mod R$ is not Serre.
(As is stated in Example \ref{1}, the subcategory $\cof_I(R)$ of $\Mod R$ is not even abelian.)
\item
It follows from Proposition \ref{12} that the subcategory $\cof_I^0(R)$ of $\Mod R$ is not abelian.
\end{enumerate}
\end{rem}

We close the section by presenting a natural question arising from Proposition \ref{12}.

\begin{ques}
Does the converse of Proposition \ref{12} hold?
To be precise, suppose that $\h_I^i(M)$ is $I$-cofinite for all finitely generated $R$-modules $M$ and all integers $i$.
Is then $\cof_I^0(R)$ an abelian subcategory of $\Mod R$\,?
\end{ques}

%%%%%%%%%%%%%%%%%%%%%%%%%%%%%%
\section{The subcategory $\c_I(R)$ of $\mod R$}

In this section, applying the results obtained in the previous section, we study the structure of the subcategory $\c_I(R)$ of the category $\mod R$ of finitely generated $R$-modules, whose definition and a couple of whose basic properties are stated below.

\begin{dfn}
We denote by $\c_I(R)$ the subcategory of $\mod R$ consisting of finitely generated $R$-modules $M$ such that the $R$-module $\h_I^i(M)$ is $I$-cofinite for all integers $i$.
\end{dfn}

\begin{rem}\label{10}
\begin{enumerate}[(1)]
\item
The subcategory $\c_I(R)$ of $\mod R$ is closed under finite direct sums and direct summands.
Indeed, Proposition \ref{2} says that the subcategory $\cof_I(R)$ of $\Mod R$ is closed under finite direct sums and direct summands, from which one can easily deduce the assertion.
\item
Question \ref{17} is equivalent to asking whether the abelianity of the subcategory $\cof_I(R)$ of $\Mod R$ implies the equality $\c_I(R)=\mod R$.
\end{enumerate}
\end{rem}

We give a connection of the subcategory $\c_I(R)$ of $\mod R$ with the subcategory $\cof_I(R)$ of $\Mod R$. 

\begin{prop}\label{14}
Suppose that $\cof_I(R)$ is an abelian subcategory of $\Mod R$.
Then $\c_I(R)$ is a thick subcategory of $\mod R$.
\end{prop}

\begin{proof}
Let $0\to L\xrightarrow{f}M\xrightarrow{g}N\to0$ be a short exact sequence of finitely generated $R$-modules.
Then for each integer $i$ there exists an exact sequence
$$
\h_I^{i-1}(M)\xrightarrow{\h_I^{i-1}(g)}\h_I^{i-1}(N)\xrightarrow{\delta^{i-1}}\h_I^i(L)\xrightarrow{\h_I^{i}(f)}\h_I^i(M)\xrightarrow{\h_I^{i}(g)}\h_I^i(N)\xrightarrow{\delta^i}\h_I^{i+1}(L)\xrightarrow{\h_I^{i+1}(f)}\h_I^{i+1}(M).
$$
Assume that $M,N$ are in $\c_I(R)$.
Then for each $i$ the four $R$-modules $\h_I^{i-1}(M)$, $\h_I^{i-1}(N)$, $\h_I^i(M)$ and $\h_I^i(N)$ belong to $\cof_I(R)$.
Since $\cof_I(R)$ is abelian, it is closed under kernels and cokernels.
Hence $\cok\h_I^{i-1}(g)$ and $\ker\h_I^{i}(g)$ are in $\cof_I(R)$.
By Proposition \ref{2}, the subcategory $\cof_I(R)$ of $\Mod R$ is thick, so in particular it is closed under extensions.
The exact sequence
$$
0\to\cok\h_I^{i-1}(g)\to\h_I^i(L)\to\ker\h_I^{i}(g)\to0
$$
is induced, which implies that $\h_I^i(L)$ is in $\cof_I(R)$.
It follows that $L$ belongs to $\c_I(R)$.
Thus, $\c_I(R)$ is closed under kernels of epimorphisms.
An analogous argument shows that $\c_I(R)$ is closed under cokernels of monomorphisms and extensions.

Thus, $\c_I(R)$ satisfies the two-out-of-three property.
It follows from this together with Remark \ref{10}(1) that $\c_I(R)$ is a thick subcategory of $\mod R$.
\end{proof}

To prove the main result of this section, we need some preparations.
We first recall the definitions of certain loci in $\spec R$ and a certain invariant for finitely generated $R$-modules.

\begin{dfn}
Let $M$ be a finitely generated $R$-module.
\begin{enumerate}[(1)]
\item
The set of prime ideals $\p$ of $R$ such that the $R_\p$-module $M_\p$ is nonfree is called the {\em nonfree locus} of $M$ and denoted by $\nf(M)$.
The set of prime ideals $\p$ of $R$ such that the $R_\p$-module $M_\p$ has infinite projective dimension is called the {\em infinite projective dimension locus} of $M$ and denoted by $\ipd(M)$.
%It is clear that the inclusion $\nf(M)\supseteq\ipd(M)$ holds.
%Both $\nf(M)$ and $\ipd(M)$ are closed subsets of $\spec R$ in the Zariski topology; see \cite{}.
\item
We denote by $\rfd_RM$ the {\em (large) restricted flat dimension} of $M$, which is defined by
$$
\rfd_RM=\sup_{\p\in\spec R}\{\depth R_\p-\depth M_\p\}.
$$
It holds that $\rfd_RM\in\Z_{\ge0}\cup\{-\infty\}$, and $\rfd_RM=-\infty$ if and only if $M=0$; see \cite[Theorem 1.1]{AIL} and \cite[Proposition (2.2) and Theorem (2.4)]{CFF}.
\end{enumerate}
\end{dfn}

Next we recall the definition of a syzygy.
Let $M$ be a finitely generated $R$-module.
Let
$$
\cdots\xrightarrow{\partial_{i+1}}P_i\xrightarrow{\partial_{i}}
P_{i-1}\xrightarrow{\partial_{i-1}}
\cdots\xrightarrow{\partial_{2}}
P_1\xrightarrow{\partial_{1}}
P_0\xrightarrow{\partial_{0}}
M\to0
$$
be an exact sequence of finitely generated $R$-modules such that each $P_i$ is projective.
Then the image of $\partial_i$ is called the {\em $i$th syzygy} of $M$, and denoted by $\syz^iM$.
Note that it is uniquely determined by $M$ and $i$ up to projective summands, that is, if $X$ and $Y$ are the $i$th syzygies of $M$, then $X\oplus P\cong Y\oplus Q$ for some finitely generated projective $R$-modules $P$ and $Q$.

For a subset $\Phi$ of $\spec R$, we denote by $\nf^{-1}(\Phi)$ and $\ipd^{-1}(\Phi)$ the subcategories of $\mod R$ consisting of finitely generated $R$-modules $M$ such that $\nf(M)$ and $\ipd(M)$ are contained in $\Phi$, respectively.
The following lemma is necessary in the proof of our main result in this section.

\begin{lem}\label{13}
\begin{enumerate}[\rm(1)]
\item
Let $\Phi$ be a subset of $\spec R$.
Let $M$ be a nonzero finitely generated $R$-module.
Put $r=\rfd_RM$.
If $M$ belongs to $\ipd^{-1}(\Phi)$, then $\syz^rM$ belongs to $\nf^{-1}(\Phi)$.
\item
There is an inclusion $\supp^{-1}(\V(I))\cap\mod R\subseteq\c_I(R)$.
\end{enumerate}
\end{lem}

\begin{proof}
(1) The assertion follows from the equality $\ipd(M)=\nf(\syz^rM)$ shown in \cite[Remark 10.2(4)]{dlr}.

(2) Let $M$ be a finitely generated $R$-module such that $\supp M\subseteq\V(I)$.
Then $\h_I^0(M)=M$ and $\h_I^i(M)=0$ for all $i>0$; see \cite[Proposition 3.2(1a)]{cd}.
Hence for all integers $i$, the $R$-module $\h_I^i(M)$ is finitely generated, which implies that $\h_I^i(M)$ is $I$-cofinite (see Remark \ref{8}(2)).
Therefore, $M$ belongs to $\c_I(R)$.
\end{proof}

We also need the notion of resolving subcategories of finitely generated modules.

\begin{dfn}\label{20}
\begin{enumerate}[(1)]
\item
A {\em resolving subcategory} of $\mod R$ is defined to be a subcategory $\X$ of $\mod R$ which satisfies the following three conditions.
\begin{enumerate}[(i)]
\item
$\X$ contains the finitely generated projective $R$-modules.
\item
$\X$ is closed under direct summands and extensions.
\item
$\X$ is closed under kernels of epimorphisms.
\end{enumerate}
\item
For a subcategory $\C$ of $\mod R$ we denote by $\res\C$ the {\em resolving closure} of $\C$, that is to say, the smallest resolving subcategory of $\mod R$ containing $\C$.
\end{enumerate}\end{dfn}

\begin{rem}\label{21}
In Definition \ref{20}(1), condition (i) can be replaced with the condition that $\X$ contains $R$.
In fact, a finitely generated projective $R$-module is a direct summand of a finite direct sum of copies of $R$, and if $\X$ is closed under extensions, then it is closed under finite direct sums.
\end{rem}

We denote by $\sing R$ the {\em singular locus} of $R$, that is, the set of prime ideals $\p$ of $R$ such that the local ring $R_\p$ is not regular.
Now we are ready to prove the following theorem, which is the main result of this section.

\begin{thm}\label{6}
Suppose that $R$ belongs to $\c_I(R)$ and that $\cof_I(R)$ is an abelian subcategory of $\Mod R$.
\begin{enumerate}[\rm(1)]
\item
One has that $\ipd^{-1}(\V(I))$ is contained in $\c_I(R)$.
\item
Suppose that $\sing R$ is contained in $\V(I)$.
Then $\c_I(R)$ coincides with $\mod R$, that is to say, $\h_I^i(M)$ is $I$-cofinite for every integer $i$ and every finitely generated $R$-module $M$.
\end{enumerate}
\end{thm}

\begin{proof}
(1) Let $M$ be finitely generated $R$-module with $M\in\ipd^{-1}(\V(I))$.
Put $r=\rfd_RM$.
It holds that
\begin{equation}\label{15}
\syz^rM
\in\nf^{-1}(\V(I))
=\res(\supp^{-1}(\V(I))\cap\mod R)
\subseteq\c_I(R).
\end{equation}
In fact, the containment and the equality in \eqref{15} are consequences of Lemma \ref{13}(1) and \cite[Theorem 4.1(1)]{kos}, respectively.
Since $\cof_I(R)$ is an abelian subcategory of $\Mod R$, it follows from Proposition \ref{14} that $\c_I(R)$ is a thick subcategory of $\mod R$.
As $R$ belongs to $\c_I(R)$ by assumption, we see from Remark \ref{21} that $\c_I(R)$ is a resolving subcategory of $\mod R$.
The inclusion in \eqref{15} is now shown by Lemma \ref{13}(2).

By \eqref{15}, the syzygy $\syz^rM$ is in $\c_I(R)$.
Since $\c_I(R)$ contains any finitely generated projective $R$-module and is closed under cokernels of monomorphisms, it is seen by descending induction on $r$ that $M$ is in $\c_I(R)$.

(2) Let $M$ be a finitely generated $R$-module.
If $\p$ is a prime ideal of $R$ such that $R_\p$ is a regular local ring, then the $R_\p$-module $M_\p$ has finite projective dimension.
This shows that $\ipd(M)$ is contained in $\sing R$, and the latter set is contained in $\V(I)$ by assumption.
Therefore, the module $M$ belongs to $\ipd^{-1}(\V(I))$.
By virtue of (1), we conclude that $M$ is in $\c_I(R)$.
Now we obtain the equality $\c_I(R)=\mod R$.
\end{proof}

\begin{rem}
Hartshorne's Example \ref{1} satisfies both of the assumptions in Theorem \ref{6} that $R\in\c_I(R)$ and $\sing R\subseteq\V(I)$.
In fact, the former statement has been verified in Example \ref{1}.
As for the latter, since the ring $R$ is regular, $\sing R$ is an empty set, so that it is contained in $\V(I)$.
\end{rem}

Now we can complete the proofs of our main theorems stated in Section 1.

\begin{proof}[Proof of Theorem \ref{3}]
The assertion follows from Propositions \ref{11}, \ref{12} and Theorem \ref{6}(2).
\end{proof}

\begin{proof}[Proof of Theorem \ref{16}]
The assertion is a direct consequence of Theorem \ref{6}(1).
\end{proof}

Recall that a local ring $(R,\m)$ is called an {\em isolated singularity} if the localization $R_\p$ is a regular local ring for all prime ideals $\p$ of $R$ with $\p\ne\m$.
Here is an application of the above theorem.

\begin{cor}
Let $R$ be an isolated singularity such that $\h_I^i(R)$ is $I$-cofinite for all $i\in\Z$.
If $\cof_I(R)$ is an abelian subcategory of $\Mod R$, then $\h_I^i(M)$ is $I$-cofinite for all finitely generated $R$-modules $M$ and all $i\in\Z$.
\end{cor}

\begin{proof}
If $I$ is a unit ideal of $R$, then $\h_I^i(M)=0$ for all $R$-modules $M$ and all integers $i$, and the conclusion obviously holds.
Hence we may assume that $I$ is contained in the maximal ideal $\m$ of $R$, and then it holds that $\sing R\subseteq\{\m\}\subseteq\V(I)$.
Now the assertion follows from Theorem \ref{6}(2).
\end{proof}

In view of Theorems \ref{3}(2)(3) and \ref{16}, we may wonder how strong the assumption that $R\in\c_I(R)$ is.
Needless to say, this is at least a necessary condition for the equality $\c_I(R)=\mod R$ to hold, but it would be reasonable to state it explicitly as a question.

\begin{ques}\label{5}
Let $I$ be an ideal of $R$.
When is $\h_I^i(R)$ an $I$-cofinite module for all integers $i$? 
\end{ques}

\begin{rem}
\begin{enumerate}[(1)]
\item
It is proved in \cite[Theorem 2]{B} that $\h_I^i(M)$ is $I$-cofinite for all integers $i$ and all finitely generated $R$-modules $M$ if and only if $\h_I^i(R)$ is $I$-cofinite and has Krull dimension at most $1$ for all $i\ge2$.
\item
It does not necessarily hold that $\h_I^i(R)$ is $I$-cofinite for all $i$, even if $R$ is a regular local ring.
Indeed, for example, consider the formal power series ring $R=k[\![x,y,z]\!]$ over a field $k$, and the ideal $I=(xy,xz)$ of $R$.
Then it follows from \cite[Theorem 3.8]{BN} that $\Hom_R(R/I,\h_I^2(R))$ is not a finitely generated $R$-module.
In particular, the $R$-module $\h_I^2(R)$ is not $I$-cofinite.
\end{enumerate}
\end{rem}

%%%%%%%%%%%%%%%%%%%%%%%%%%%%%%%%%%%
\begin{ac}
The authors thank Kamal Bahmanpour, Ken-ichiroh Kawasaki and Takeshi Yoshizawa for valuable, useful and helpful comments.
\end{ac}
%%%%%%%%%%%%%%%%%%%%%%%%%%%%%%%%%%%%%

%%%%%%%%%%%%%%%%%%%%%%%%%%%%%%%%%%%%%%%%%%%%%%%%%
\end{document}